\documentclass[preprint,12pt,authoryear]{elsarticle}

\usepackage{amssymb}
 \usepackage{amsthm}
 \usepackage{amsmath}
 \usepackage{ulem}
\usepackage{graphicx}
\usepackage[matrix,arrow,curve]{xy}
\usepackage{lscape}

\newtheorem{definition}{Definition}
\newtheorem{remark}{Remark}
\newtheorem{proposition}{Proposition}
\newtheorem{example}{Example}

\newtheorem{lemma}{Lemma}
\newtheorem{theorem}{Theorem}

\journal{Topology and its Applications}

\begin{document}

\begin{frontmatter}

%% Title, authors and addresses

%% use the tnoteref command within \title for footnotes;
%% use the tnotetext command for theassociated footnote;
%% use the fnref command within \author or \address for footnotes;
%% use the fntext command for theassociated footnote;
%% use the corref command within \author for corresponding author footnotes;
%% use the cortext command for theassociated footnote;
%% use the ead command for the email address,
%% and the form \ead[url] for the home page:
\title{Quasirational relation modules and $p$-adic  Malcev completions.\tnoteref{label1}}
%% \tnotetext[label1]{}
\author{Andrey Mikhovich\corref{cor1}\fnref{label2}}
\ead{mikhandr@mail.ru}
%% \ead[url]{home page}
%% \fntext[label2]{}
%% \cortext[cor1]{}
%%\address{Belarusian State University\fnref{label3}}
%% \fntext[label3]{}

\title{Quasirational relation modules and $p$-adic Malcev completions.}

%% use optional labels to link authors explicitly to addresses:
%% \author[label1,label2]{}
%% \address[label1]{}
%% \address[label2]{}

%%\author{Andrey Mikhovich}

\address{Moscow State University}
\address{Moscow, Russia}

\begin{abstract}
We introduce the concept of quasirational relation modules for discrete (pro-$p$) presentations of discrete
(pro-$p$)
groups. It is shown, that this class of presentations for discrete groups contains CA-presentations and their subpresentations. For pro-$p$-groups we see
that all presentations of pro-$p$-groups with a single defining relation are quasirational. We offer definitions of $p$-adic
$G(p)$-completion and $p$-adic rationalization of relation modules which are adjusted to quasirational pro-$p$-presentations. $p$-adic rationalizations of quasirational relation modules of pro-$p$-groups are isomorphic to $\mathbb{Q}_p$-points of abelianized $p$-adic Malcev completions.
\end{abstract}


\end{frontmatter}

%% \linenumbers

%% main text
\section{Introduction and motivations.}
Whitehead asphericity question is one of the oldest problems  in Combinatorial Group  Theory. Similarly, validity of analogs of Lyndon Identity Theorem in Combinatorial Theory of pro-$p$-Groups \cite[10.2]{Se2} is still far to be understood.
We introduce a class of finite presentations ("quasirational" presentations) which contains aspherical presentations as well as their subpresentations. In a pro-$p$ case (pro-$p$-groups are projective limits of finite $p$-groups with its pro-$p$-presentations, every pro-$p$-group could be presented as some factor of free pro-$p$-group with an appropriate space of topological generators by a certain closed normal subgroup (see \cite{Se3})) our concept includes one-relator pro-$p$-groups.

Permutational features of relation modules plays the key role in asphericity type problems  \cite{CCH,Mel}. Moreover  permutationality unlike asphericity behaves properly with respect to coinvariant completions of relation modules and holds by their scalar extensions. We will see that quasirational presentations  may be studied by passing to rationalized completions $\overline{R}\widehat{\otimes} \mathbb{Q}_p: =\varprojlim R/[R,R\mathcal{M}_n]\otimes \mathbb{Q}_p$ (since $\varprojlim$ is left exact for quasirational pro-$p$-presentations we have an embedding of abelian groups $\overline{R}\hookrightarrow \overline{R}\widehat{\otimes} \mathbb{Q}_p$)  in a spirit of Gasch\"{u}tz theory (see \cite{Gru}).	$\overline{R} \widehat{\otimes} \mathbb{Q}_p$ has a structure of topological $\mathcal{O}(F_u)^*-$module in a sense of \cite{Ha2}, where $F_u$ be a free prounipotent group with a complete Hopf algebra $\mathcal{O}(F_u)^*:=Hom_k(\mathcal{O}(F_u),k)$  (the "coordinate ring" of $F_u$ i.e. dual to the representing Hopf algebra $\mathcal{O}(F_u)$ of $F_u$ (\cite[3]{Vez})).  The same structure  is valid for quasirational presentations of pro-$p$-groups.

As we will show,  there is an isomorphism  $\overline{R}\widehat{\otimes} \mathbb{Q}_p\cong \overline{R^{\wedge}_u}(\mathbb{Q}_p)$ in the category of topological  $\mathcal{O}(F_u)^*$- modules, where  $\overline{R^{\wedge}_u}(\mathbb{Q}_p)$ is a certain prounipotent module over the same complete Hopf algebra. The latter category seems extremely convenient for the required calculations.
\section{Quasirational relation modules.}
\textbf{For pro-$p$-groups, fix a prime $p>0$ throughout the paper}
(see \cite
{Se3} for details on pro-$p$-groups). \textbf{
For  discrete groups, $p$ will vary}. Let $G$ be a (pro-$p$)group which has a (pro-$p$)presentation of finite type
\begin{eqnarray}
1 \rightarrow R \rightarrow F \rightarrow G \rightarrow 1 \label{1}
\end{eqnarray}

Let $\overline{R}=R/[R,R]$ be the corresponding relation $G$-module, where $[R,R]$ is a (closed) commutator subgroup (in the
pro-$p$-case). Then denote $\mathcal{M}_n$ the corresponding Zassenhaus $p$-filtration of $F,$ which is defined by the rule $\mathcal{M}_n=\{f \in F\mid f-1 \in {\Delta}^n_p, \quad {\Delta}_p =ker\{\mathbb{F}_pF \rightarrow \mathbb{F}_p\}\}$ (see \cite[7.4]{Koch} for details).

\begin{definition} A presentation \ref{1} is \textbf{quasirational} if for every $n>0$ and each prime $p>0$ the $F/R\mathcal{M}_n$-module $R/[R,R\mathcal{M}_n]$ has no $p$-torsion ($p$ is fixed for pro-$p$-groups and run all primes $p>0$ and corresponding $p$-Zassenhaus filtrations in discrete case). The relation modules of such presentations will be called \textbf{quasirational relation modules}.
\end{definition}

\begin{example} Let $\overline{R}$ is a (pro-p-)permutational $G$-module, so by definition $\overline{R}$ has a permutational $G$-basis as a  projective
$(\mathbb{Z}_p)$$\mathbb{Z}$-module. Then, obviously, $\overline{R}$ is quasirational.
\end{example}

\begin{remark} The definition of quasirationality is rigid. Since $R/[R,F]$ has no torsion, then the Schur multiplyer has no torsion.
As a consequence we see that $\mathbb{Z}\oplus \mathbb{Z}/pZ$ has no quasirational presentations. If $G$ is a finite $p$-group, then $H_2(G)$ has a
finite exponent and is not trivial while $G$ is not a cyclic (since its $mod(p)$ factor is not trivial). Hence the only finite $p$-groups, which
have quasirational presentations are cyclic. This shows that quasiratinal finite $p$-groups match with finite aspherical
pro-$p$-groups \cite[Theorem 2.7]{Mel}.
\end{remark}

We need the following elementary
\begin{lemma} Let $G$ be a finite $p$-group acting on a finite abelian group $M$ of exponent $p$. Then the factor module of
coinvariants $M_G=M/(g-1)M\neq 0$, where $(g-1)M$ is a submodule of $M$, generated by elements of the form $(g-1)m$, where $g \in G, m \in
M$.
\end{lemma}

\begin{proof} We prove by induction on a rank $n$ of $M$. If $n=1$ then $M=\mathbb{Z}/p\mathbb{Z}$ is a trivial $G$-module, since $(\mid Aut(\mathbb{Z}/p\mathbb{Z})\mid,p)=(\mid\mathbb{Z}/(p-1)\mathbb{Z}\mid,p)=1$, so
$M_G=M\neq0$. Let $n=k$, then a submodule of $G$-fixed elements of $M$ is not trivial \cite[Ch IX, 1]{Se1} and contains $M_0=\mathbb{Z}/p\mathbb{Z}$  which has trivial $G$-action.
Let $M_1=M/M_0$ and $\psi:M\rightarrow M_1$ is the corresponding homomorphism of factorization. Since
$\psi((g-1)M)=(g-1)M_1$,  $\psi$ induces the surjection $M_G \twoheadrightarrow (M_1)_G$.
But $(M_1)_G\neq 0$ by induction hence $M_G\neq0$.
\end{proof}

 There is an old problem due to \cite[10.2]{Se2} concerning the description of relation modules of pro-$p$-groups with a single defining relation (i.e. $dim_{\mathbb{F}_p}H^2(G,\mathbb{F}_p)=1$ \cite[4.3]{Se3}). We have:

\begin{proposition} Suppose \eqref{1} is a presentation of a pro-$p$-group $G$ with a single defining relation, then \eqref{1} is quasirational.
\end{proposition}
\begin{proof} Note that $R/[R,RF]=R/[R,F]=\mathbb{Z}_p,$ where $R=(r)_F$ is the topological normal closure of $r\in Fr(F)$, $Fr(F)$ is the Frattini subgroup of a free pro-$p$-group $F$. Indeed, let $\phi:R \rightarrow \mathbb{Z}_p$ be a projection of $R$ onto a free cyclic subgroup generated by $r$. Since we can take a basis (a profinite subspace) of $R$ as a free pro-$p$-group consisting of conjugates of $r$, i.e. $r^f, f \in F$, there  exists a surjection $R/[R,F]
\twoheadrightarrow \mathbb{Z}_p$. However $R/[R,F]$ is generated by the image of $r$, i.e. $\overline{r} \in R/[R,F]=\langle r\rangle=\mathbb{Z}_p$.  Suppose
$\overline{R}_n=R/[R,R\mathcal{M}_n]$ has torsion, then $R/[R,R\mathcal{M}_n]=M_{tors} \bigoplus
M_{\mathbb{Z}_p}$ where $M_{tors}$ is a torsion subgroup, $M_{\mathbb{Z}_p}$ is a projective $\mathbb{Z}_p$-module. $M_{tors}$ and $M_{\mathbb{Z}_p}$ are $\mathbb{Z}_p[G_n]$-submodules, where
$G_n=F/\mathcal{M}_nR$. Consider $mod(p)$ reduction of  $\overline{R}_n$, then $\overline{R}_n/{p\overline{R}_n}$ as a $\mathbb{F}_p[G/\mathcal{M}_nR]$-module has a decomposition $\overline{R}_n/{p\overline{R}_n}=M_{tors}/pM_{tors}\oplus M_{\mathbb{Z}_p}/pM_{\mathbb{Z}_p}$. $M_{tors}\neq 0$,
hence $M_{tors}/pM_{tors}\neq 0$. Now $M_{tors}/pM_{tors}$ is an abelian group of exponent $p$, $G_n$ is a finite
$p$-group which acts on a finite abelian group of exponent $p$ so by
the lemma $(M_{tors}/pM_{tors})_G\neq 0$.  Since
$(M_{\mathbb{Z}_p})_G=\mathbb{Z}_p$ we have $dim_ {\mathbb{F}_p}((M_{\mathbb{Z}_p}/pM_{\mathbb{Z}_p})_G\oplus (M_{tors}/pM_{tors})_G) =
dim_{\mathbb{F}_p}(R/R^p[R,F])=dim_{\mathbb{F}_p}H^2(G,\mathbb{F}_p)=1$ \cite[4.3]{Se3}. Therefore $dim_{\mathbb{F}_p}((M_{\mathbb{Z}_p}
/pM_{\mathbb{Z}_p})_G)=1$. Hence we have a contradiction $dim_{\mathbb{F}_p}(M_{tors}/pM_{tors})_G=0$ and $R/[R,R\mathcal{M}_n]$ has no torsion.
\end{proof}

Let \eqref{1} be a presentation of a discrete group. Then $K(X;R)$ denotes the standard two-dimensional $CW$-complex of \ref{1} (see \cite{CCH} for details). We call $\langle X;R\rangle$ \textbf{aspherical} if $K(X;R)$ is aspherical, i.e. $\pi_q(K(X;R))=0$, $q\geq 2$. Simple asphericity is in a way a rather restrictive concept and therefore needs to be extended. There is an attractive setting of CA(combinatoricaly aspherical)-presentations \cite[1]{CCH} which can be characterized by the following:
\begin{proposition}
Let $\langle X;R\rangle$ be a presentation of $G$ and suppose every element of $R$ is reduced. Let $\overline{N}$ be the corresponding relation module. Then $\langle X;R\rangle$ is CA and concise if and only if $\overline{N}$ decomposes as a $\mathbb{Z}G$-module into a direct sum of cyclic submodules $P_r,r\in R$, where $P_r$ is generated by $\overline{r}=rN'$ subject to the single relation $\pi(s)\overline{r}=\overline{r},s$ being the root of $r$.
\end{proposition}
\begin{proof}
\cite[1.2]{CCH}
\end{proof}
\begin{proposition} Let \eqref{1} be a discrete CA-presentation (so it has a permutation relation module), then \eqref{1} and all subpresentations of \eqref{1} are quasirational.
\end{proposition}

\begin{proof} Let $R_{-1}$ be a presentation without any one relator (excluded from $R$). First prove the statement for $R_{-1}$, in general the proof will be the same. First note that
$R_{-1}/[R_{-1},F]=\mathbb{Z}^{r-1}$ since its coinvariants could be mapped into coinvariants of a permutation module (see the previous proof); moreover there is a homomorphism of abelian groups $\phi:R_{-1}/[R_{-1},F] \rightarrow R/[R,F]$ and
since $R/[R,F]$ is a free abelian group, the image is just a copy of $\mathbb{Z}^{r-1}$.
 Indeed $R_{-1}/[R_{-1},F]$ as an abelian group is generated by $(r-1)$ elements (so it must be isomorphic to
 $\mathbb{Z}^{r-1}$). If the abelian group $R_{-1}/[R_{-1},R_{-1}\mathcal{M}_n]$ for some p-Zassenhaus filtration $\mathcal{M}_n$ has torsion then torsion elements $M_{tors}$ must generate a direct $G_n$-module
 summand and we use  mod(p) reduction and the lemma to see that the coinvariants of $M_{tors}/pM_{tors}$ will bring a new non zero torsion summand in $R_{-1}/[R_{-1},F]$. Hence we see that $R_{-1}/[R_{-1},R_{-1}\mathcal{M}_n]$ has no torsion.
\end{proof}

\begin{remark} CA-presentations and pro-$p$-aspherical relation modules \cite[2]{Mel} are the main motivations for our concept of quasirationality. Their remarkable features were understood in \cite{CCH}(for discrete presentations) and \cite{Mel} (for pro-$p$-groups).
\end{remark}

\begin{proposition}Let \eqref{1} is a (pro-$p$-)presentation, then the following conditions are equivalent:

(i) \eqref{1} is quasirational

(ii) abelian (pro-$p$-)group $R/[R,F]$ has no torsion
\end{proposition}

\begin{proof}
See the previous proofs.
\end{proof}

\begin{proposition}

(i) Let \eqref{1} is an aspherical (pro-$p$-)presentation \cite{Mel}, then \eqref{1} and all its subpresentations are quasirational.

(ii) Let \eqref{1} be a quisirational presentation of a pro-$p$-group $G$, then all presentations of $G$ are quasirational
\end{proposition}

\begin{remark} Let $v$ is a generator of the $mod(p)$ relation
module $\overline{R}/p\overline{R}$ of a one-relator pro-$p$-group. Suppose $H$ is the stabiliser of $v$ by the action of $G$. Then there is a surjection $\eta: \mathbb{F}_p[[G/H]]\twoheadrightarrow
\overline{R}/p\overline{R}$ and we wants to check when $\eta$ is isomorphism. It's enough to establish
injectivity of the map $gr(\eta): gr_{\Delta}\mathbb{F}_p[[G/H]] \twoheadrightarrow gr_{\Delta}\overline{R}/p\overline{R}$. Then the theorem due to Alperin \cite[Corollary]{Sha} gives a criterion of injectivity.
\end{remark}

\section{The $p$-adic rationalization of quasirational relation modules.}

For every relation module $\overline{R}$,  similarly to \cite[8]{Bou}, we introduce the notions  of $p$-adic $G(p)$-completion $\overline{R}_c$ and $p$-adic rationalization:

\begin{definition}. Let $\overline{R}$ be a quasirational relation module, let $\overline{R}_c$ be its corresponding \textbf{$p$-adic $G(p)$-completion}. Set
$$ \overline{R}_c= \varprojlim {R/[R,R\mathcal{M}_n]} \otimes_{\mathbb{Z}} \mathbb{Z}_p=\varprojlim
(R/[R,R\mathcal{M}_n])^{\wedge},$$ where $(R/[R,R\mathcal{M}_n])^{\wedge}$ is just a pro-$p$-completion of a finitely generated free abelian group
($\overline{R}_c\cong \overline{R}$ in a case of pro-$p$-presentations). Define
\textbf{$p$-adic rationalization} as a pro-$fd$ (finite
dimensional)-module over pro-$p$-completion $G^{\wedge}$ of $G$ \cite[1.4]{Se3} which is by definition
$$\overline{R}_c\widehat{\otimes}_{\mathbb{Z}_p}\mathbb{Q}_p= \varprojlim (R/[R,R\mathcal{M}_n])^{\wedge} \otimes_{\mathbb{Z}_p}\mathbb{Q}_p.$$
\end{definition}

If \eqref{1} is some quasirational presentation, then we can look at topological closure $\widehat{R}$ of $R$ in the pro-$p$-topology of $F$, i.e.
\begin{eqnarray}
\widehat{R}\cong \varprojlim_{U_o \lhd F}R/{R\cap U}, \mid F:U\mid=p^s\label{2}.
\end{eqnarray}
In general $R$ could have many more ``open" subgroups than those which arise as intersections $R\cap U$ where $U_o\unlhd F$, hence the behavior of $\overline{\widehat{R}}$ in general differs from that of $\overline{R_c}$.

Let's explore carefully quasirational pro-$p$-presentations of finite type. First note that the group of $\mathbb{Q}_p$-points of any affine group scheme $G$ has the $p$-adic topology. Indeed, \cite[Chap.2]{DM} shows that $G$ can be expressed as a filtered inverse limit
$$G=\varprojlim G_{\alpha}$$
of linear algebraic groups. Each $G_\alpha(\mathbb{Q}_p)$ has a canonical $p$-adic topology induced by the embeding $G_{\alpha}\hookrightarrow GL_n$. Define the topology on $G(\mathbb{Q}_p)$ by
$$G(\mathbb{Q}_p)=\varprojlim G_{\alpha}(\mathbb{Q}_p).$$

\begin{definition}
Fix a pro-p-group $G$. Define $p$-adic Malcev completion of $G$ by a universal diagram where $\rho$ is continuous Zarissky-dense homomorphism of $G$ into $\mathbb{Q}_p$-points of a prounipotent affine group $G_u^{\wedge}(\mathbb{Q}_p)$
$$\xymatrix@R=0.5cm{
                &         G^{\wedge}_u(\mathbb{Q}_p)  \ar[dd]^{\tau}     \\
 G \ar[ur]^{\rho} \ar[dr]_{\chi}                 \\
                &         H(\mathbb{Q}_p)              }$$
We require that for every continuous Zarissky-dense homomorphism $\chi$ into $\mathbb{Q}_p$-points of a prounipotent affine group $H$ there is a unique homomorphism $\tau$ of prounipotent groups, making the diagram commutative.
\end{definition}

Such an object always exists \cite[Prop.2.1]{Ha1};\cite{Knu, HM}.

Take a presentation \eqref{1}. Since $F$ is a free pro-$p$-group, it embeds into $p$-adic Malcev completion $F_u$ and we can form $R_u \lhd F_u$, the Zarissky closure of $R$ in $F_u$. The universal property of $R_u^{\wedge}$ defines a homomorphism $\tau$
$$\xymatrix@R=0.5cm{
                &         R_u^{\wedge} \ar@{->>}[dd]^{\tau}     \\
  R \ar[ur] \ar[dr]_{i}                 \\
                &         R_u                 }$$
which must be onto since $im(\tau)\supseteq R$ ($R$ is dense in $R_u$). Passing to abelianizations we obtain a homomorphism of abelian prounipotent groups $\overline{\tau}: \overline{R_u^{\wedge}} \to \overline{R_u}$

In a forthcoming paper we include $\overline{\tau}$ into a sequence of prounipotent abelian groups in a way similar to \cite[Prop.7, Cor.]{BH},

\begin{equation}
\xymatrix{
0 \ar[rr] & &\varpi_2  \ar[rr] & & \overline{C_u} \ar[rr] \ar[dr]_{\eta}
                &  &    \overline{R_u} \ar[rr] & & 0  \\
            & & & &   &  \overline{R_u^{\wedge}}  \ar[ur]_{\overline{\tau}}      } \label{3}
\end{equation}
where $\eta$ could be studied, completing fibrations \cite[3.4]{Pri}. The diagram \eqref{3} implies the sequence of $\mathbb{Q}_p$-points
$$\xymatrix{
 \overline{C_u}(\mathbb{Q}_p) \ar@{->>}[rr] \ar@{->>}[dr]_{\eta(\mathbb{Q}_p)}
                &  &    \overline{R_u}(\mathbb{Q}_p)  \\
              &  \overline{R_u^{\wedge}}(\mathbb{Q}_p)  \ar@{->>}[ur]_{\overline{\tau}(\mathbb{Q}_p)}      } $$

But the starting point and motivation is here:

\begin{theorem} Let \eqref{1} be a quasirational pro-$p$-presentation, then there is an isomorphism of topological $\mathcal{O}(F_u)^{\ast}-$ modules of finite type \cite[3]{Ha2} (we use a slightly different than \cite[2.5]{Ha2} topology on tensor products to obtain continuous actions of $\mathcal{O}(F_u)^{\ast}$)

$$\psi:{\overline{R}\widehat{\otimes}_{\mathbb{Z}_p}\mathbb{Q}_p \cong \overline{R_u^{\wedge}}(\mathbb{Q}_p)}$$
\end {theorem}
\begin{proof}(just a basic construction of $\psi$)
 We prove that $\overline{R}_c \widehat{\otimes}_{\mathbb{Z}_p} \mathbb{Q}_p$ has the universal property of $p$-adic Malcev completion with respect to dense continuous homomorphisms of pro-p-group $R$ into abelian pro-unipotent groups. Since every pro-unipotent group is an inverse limit of some surjective inverse system of unipotent groups we need to prove the universal property just with respect to homomorphisms into abelian unipotent groups.
Let $\phi : R\rightarrow U\cong \mathbb{Q}_p^l$ be any continuous Zarissky-dense homomorphism. Since $\phi (R)$ is dense in $\mathbb{Q}_p^l$ and $\mathbb{Z}_p$ is PID then $\phi(R)\cong \mathbb{Z}_p^l$ (compare \cite[Lemma7.5]{HM}). But then quasirationality implies the existence of required $\psi : \overline{R}\widehat{\otimes}_{\mathbb{Z}_p}\mathbb{Q}_p\rightarrow \mathbb{Q}_p^l$. More exactly, let $\gamma_n:Z_p^l\twoheadrightarrow \mathbb{Z}_p^l/p^n\mathbb{Z}_p^l, W=ker(\gamma_n\circ\phi).$ By construction \eqref{2} of the topology on $R$ we can find $k\in \mathbb{N}$
such that $R \cap \mathcal{M}_k\subseteq W$ but $[R,\mathcal{M}_k]\subseteq R\cap \mathcal{M}_k.$ Take the canonical epimorphism $\xi_k: \overline{R}\twoheadrightarrow R/[R,R\mathcal{M}_k]$. Since $R/[R,R\mathcal{M}_k]$ has no torsion we know that $\exists \eta_k:R/[R,R\mathcal{M}_k]\twoheadrightarrow \mathbb{Z}_p^l$ so define $\psi=(\eta_k\circ\xi_k)\otimes id$.

Routine checking shows that $\psi$ is a $\mathcal{O}(F_u)^{\ast}-$ module homomorphism, $\mathcal{O}(F_u)^{\ast}-$ module structures on the left and on the right are defined from the actions of $F$ and $F_u$ by conjugation (details in \cite{Mikh}).
\end{proof}

%% \label{}

%% If you have bibdatabase file and want bibtex to generate the
%% bibitems, please use
%%
%% \bibliographystyle{elsarticle-harv}
%%  \bibliography{<your bibdatabase>}

%% else use the following coding to input the bibitems directly in the
%% TeX file.

\end{document}